\documentclass{amsart}
\usepackage{indentfirst,enumerate,mathrsfs,array}
\usepackage[pdftex]{hyperref}
\usepackage{amsmath,bm}

\usepackage{graphicx}
\usepackage{amssymb,amsfonts,anyfontsize}
\usepackage{calligra,epsfig}
\usepackage{array,amscd}
\usepackage{multirow}

\def\DD{\mathcal D}
\def\HH{\mathcal H}
\def\PP{\mathcal P}
\def\NN{\mathbb N}
\def\RR{\mathbb R}
\def\xx{\mathbf{x}}
\def\yy{\mathbf{y}}
\def\zz{\mathbf{z}}

\newcommand{\la}{\lambda}
\newcommand{\al}{\alpha}
\newcommand{\ella}{\ell_A}
\newcommand{\elll}{\ell_L}
\newcommand{\fz}{f_{\mathbf{z},\la}}
\newcommand{\fp}{\hat f}
\newcommand{\fr}{\fp^R}
\newcommand{\fbar}{\bar f}

\newcommand{\gp}{\hat g}

\newcommand{\LL}{\mathscr{L}^2(X,\nu;Y)}

\newcommand{\dl}{\mathcal{D}(L)}
\newcommand{\da}{\mathcal{D}(A)}
\newcommand{\dr}{d(R)}
\newcommand{\dra}{d_A(R)}
\newcommand{\drp}{d^{p}(R)}
\newcommand{\sx}{S_\xx}
\newcommand{\tx}{T_\xx}
\newcommand{\ip}{I_\nu}
\newcommand{\tp}{T_\nu}

\newcommand{\argmin}{\operatornamewithlimits{argmin}}
\newcommand{\paren}[1]{\left(#1\right)}
\newcommand{\brac}[1]{\left\{#1\right\}}
\newcommand{\sbrac}[1]{\left[#1\right]}
\newcommand{\inner}[1]{\left\langle#1\right\rangle}
\newcommand{\norm}[1]{\left\|#1\right\|}
\newcommand{\scalar}[3]{\langle{ #1},{#2} \rangle_{#3}}

\newcommand{\tr}{\operatorname{tr}}
\newtheorem{theorem}{Theorem}[section]

\newtheorem{corollary}[theorem]{Corollary}
\newtheorem{definition}[theorem]{Definition}
\newtheorem{proposition}[theorem]{Proposition}
\newtheorem{assumption}{Assumption}

\RequirePackage{color}

\allowdisplaybreaks
\begin{document}
\title{Statistical Inverse Problems in Hilbert Scales}

\author{Abhishake}
\address{Institute of Mathematics, Technical University of Berlin, Stra{\ss}e des 17. Juni 136, 10623 Berlin, Germany}
\email{abhishake@tu-berlin.de}
%\date{Version: \today}


\begin{abstract}
In this paper, we study the Tikhonov regularization scheme in Hilbert scales for the nonlinear statistical inverse problem with a general noise. The regularizing norm in this scheme is stronger than the norm in Hilbert space. We focus on developing a theoretical analysis for this scheme based on the conditional stability estimates. We utilize the concept of the distance function to establish the high probability estimates of the direct and reconstruction error in Reproducing kernel Hilbert space setting. Further, the explicit rates of convergence  in terms of sample size are established for the oversmoothing case and the regular case over the regularity class defined through appropriate source condition. Our results improve and generalize previous results obtained in related settings.
\end{abstract}
\maketitle

\section{Introduction}
In this paper, we study the nonlinear operator equation
\begin{equation}\label{Op.eqn}
A(\fp)=\gp
\end{equation}
with the infinite dimensional separable real Hilbert spaces~$\HH$ and~$\HH'$ with the inner products~$\scalar{\cdot}{\cdot}{\HH}$ and~$\scalar{\cdot}{\cdot}{\HH'}$, respectively. Here,~$\HH'$ is the space of functions between a Polish space~$X$ and a real separable Hilbert space~$Y$. We observe the noisy values of the function~$\gp$ at the inputs~$x_i$:
\begin{equation}\label{Model}
y_i=\gp(x_i)+\varepsilon_i
\end{equation}
for~$1\leq i \leq m$. Here,~$m$ is the number of observations which is called the sample size. In contrary to the direct learning scheme where we estimate the function~$\gp$, here we aim to estimate the function~$\fp$ directly from the observations. 

A common approach to stably approximate the solution of equation~\eqref{Op.eqn} is the Tikhonov regularization scheme. Sometimes, we have the information about the true solution, e.g., the true solution may be differentiable. To incorporate this information, we employ the Tikhonov regularization in Hilbert scales. This scheme consists of a functional which is a linear combination of a fidelity term measuring the fitness of the data and a penalty term in a stronger norm forcing the smoothness in the approximated solution. To define this scheme, we introduce a densely defined, unbounded, closed, linear, self-adjoint, strictly positive operator~$L : \mathcal{D}(L)\subset \HH \to \HH$ such that for some~$\elll>0$,
\begin{equation}\label{L.unbound}
 \elll\norm{f}_{\HH} \leq \norm{Lf}_{\HH} \quad\forall f \in \dl.
\end{equation}
Here, we observe that~$L^{-1}:\HH\to\HH$ is a bounded operator from the strict positivity of the operator~$L$.

The Tikhonov functional for the considered nonlinear inverse problem with sample~$\zz=\brac{(x_i,y_i)}_{i=1}^m$ is given by
$$\mathcal{E}_{\zz,\la}(f)=\brac{\frac{1}{m}\sum\limits_{i=1}^m\norm{A(f)(x_i)-y_i}_Y^2+\la\norm{L\paren{f-\fbar}}_{\HH}^2},$$
where~$\fbar\in\mathcal{D}(A)\cap\mathcal{D}(L)$ is an initial guess. The regularization parameter~$\la > 0$ has to balance both terms appropriately. Then, the Tikhonov regularization scheme in Hilbert scales can be defined as
\begin{equation}\label{fzl}
\fz = \argmin\limits_{f\in\mathcal{D}(A)\cap\mathcal{D}(L)} \mathcal{E}_{\zz,\la}(f).
\end{equation}
For the continuous and weakly sequentially closed operator~$A$, there exists a global solution of the regularization scheme in~\eqref{fzl}. But it is not necessarily unique, since~$A$ is nonlinear (see~\cite[Section 4.1.1]{Schuster2012}).

We consider the Hilbert scales~$\HH_a$ generated by the operator~$L$.  Here the spaces~$\HH_a := \mathcal{D}(L^a)$ are Hilbert spaces equipped with the inner product~$\inner{ f,g }_{\HH_a}=\inner{ L^a f,L^a g}_\HH,~~f, g \in\HH_a$. For the Hilbert scales, we have the well-known interpolation inequality
\begin{equation}\label{interpolation}
\norm{f}_{\HH_b}\leq\norm{f}_{\HH_a}^{\frac{c-b}{c-a}}\norm{f}_{\HH_c}^{\frac{b-a}{c-a}},\qquad f\in \HH_c
\end{equation}
which holds for any~$a < b < c$.

%In classical inverse problems, Hilbert scales are widely considered in regularization schemes~\cite{Engl,Nair2005}. Learning in Hilbert scales is introduced in~\cite{Rastogi2020b}. 

%Generally, the two steps approaches are considered in the classical inverse problems~\cite{Bissantz2004,Engl,Hohage,Schuster2012}. First, an approximate~$g_\delta$ of~$\gp$ is estimated from the observations~$\brac{(x_i,y_i)}_{i=1}^m$ or assumed to be known. Then, the solution of the operator equation~$A(\fp)=g_\delta$ is obtained using the regularization schemes in the second step. Some steps are taken to directly estimate the solution of~\eqref{Model} directly from the data but the error bounds does not converges to~$0$ as the sample size goes to~$\infty$~\cite{Krebs2009,Nair2007}.

The regularization schemes in Hilbert scales have been well-studied and analysed under the different assumptions in the classical inverse problems~\cite{Bissantz2004,Engl,Hohage,Schuster2012}. In learning theory, the general regularization in Hilbert scales is introduced for the linear inverse problems and established the rates of convergence~\cite{Rastogi2020b}. Nicole et al.~\cite{Mucke2020} studied the Stochastic Gradient Descent in Hilbert scales and the author provided different examples of Hilbert scales in learning. Further, the authors discussed the error estimates for the Stochastic Gradient Descent scheme for the direct learning problem. In the paper~\cite{Rastogi2020a}, the rates of convergence are established for nonlinear statistical inverse learning problems in the RKHS setting. The authors considered some assumptions on the nonlinearity of the operator~$A$ such as Fr{\'e}chet differentiability of the operator, Lipschitz continuity of the Fr{\'e}chet derivative, and the link condition to transfer the smoothness in terms of the operator~$L$ to the covariance operator. 

Here, we consider the nonlinear inverse learning problem in Hilbert scales satisfying conditional stability estimates characterized by general concave index functions. We use the Tikhonov regularization schemes to obtain the stable approximate solution in the RKHS framework. Werner and Hofmann~\cite{Werner2019a} illustrated the validity of the conditional stability estimates in different models and real-world situations. The authors showed that the derivative of~$A$ is not always necessary for this condition.

For the regularization schemes in the RKHSs, generally, we consider the smoothness by the source condition in terms of the covariance operator which implies the rates of convergence. The covariance operator depends on the considered kernel and unknown probability measure. Therefore, the source condition cannot be verified practically. Moreover, the misspecified kernel affects the source condition and consequently, the rates of convergence for the regularization schemes. Here, we consider the smoothness of the true solution in terms of the known operator~$L$ which can be checked in practice. We divide the smoothness into two cases: the regular case (i.e.,~$\fp\in \dl$) and the oversmoothing case (i.e.,~$\fp\notin \dl$). The oversmoothing case is very delicate. We consider that the regularized solution in Hilbert scales~\eqref{fzl} belongs to~$\dl$. But the true solution does not belong to~$\dl$ in the oversmoothing case. The analysis is also tricky for nonlinear inverse problems since the Tikhonov regularization in Hilbert Scales does not have an explicit solution. The analysis starts with the step~$\mathcal{E}_{\zz,\la}(\fz)\leq \mathcal{E}_{\zz,\la}(\fp)$. But~$\mathcal{E}_{\zz,\la}(\fp)$ is not well-defined in the oversmoothing case (since~$\fp\notin \dl$). We will utilize the concept of distance functions to overcome this problem. 

%This problem can be solved by considering the concept of the distance functions.
%In this paper, we consider Tikhonov regularization in Hilbert scales and analyze its convergence using the conditional stability estimates.
% corresponding to~$r\geq 1$  and~$r< 1$ 

The main results of our paper can be summarized as follows:
\begin{itemize}
\item[-] We discuss the rates of convergence for the Tikhonov regularization in Hilbert Scales under a conditional stability assumption for the inverse problem. 

\item[-] We obtain the error estimates in the absence of the widely-considered source condition. We will use the concept of the distance functions for this. 

\item[-] We establish the error bounds in both the regular case and the oversmoothing case for the appropriate benchmark smoothness. 

\end{itemize}

The manuscript is organized as follows: In Section~\ref{Sec:Notation}, we present the basic definitions, notation, and assumptions required in our analysis. In Section~\ref{Sec:Analysis}, we state and prove our main results. Here, we discuss the rates of convergence for Tikhonov regularization in Hilbert scales in the probabilistic sense. In Section~\ref{Sec:Explicit.rates}, we present the explicit rates in terms of sample size by bounding the distance functions. In Appendix, we state the probabilistic estimate of perturbation inequalities.

%We discuss  In Section~\ref{Sec:Discussion}, the presentation closes with the discussion on the relation to previous results on regularization in Hilbert scales. 

%\begin{enumerate}[(i)]
%\item For all~$x\in X$ and~$y\in Y$, the function~$K_xy=K(\cdot,x)y$, defined by
%  \[z \in X \mapsto (K_xy)(z)=K(z,x)y \in Y,\] belongs to~$\HH$; this allows us to define the linear mapping~$K_x: Y \rightarrow \HH: y \mapsto K_xy$.
%\item The span of the set~$\brac{K_xy:x\in X, y\in Y}$ is dense in~$\HH$.
%\item For all~$f\in \HH$,~$x\in X$ and~$y \in Y$,~$\inner{ f(x),y}_Y=\inner{ f,K_xy}_{\HH}$, in other words~$f(x) = K_x^* f$ (reproducing property).
%\end{enumerate}

\section{Notation and Assumptions}\label{Sec:Notation}

Let the input space~$X$ be a Polish space and the output space~$(Y, \inner{ \cdot,\cdot}_Y)$ be a real separable Hilbert space. We consider the joint probability measure~$\rho$ on the sample space~$Z=X\times Y$. We denote the marginal distribution on~$X$ by~$\nu$ and the conditional distribution of~$y$ given~$x$ by~$\rho(y|x)$. Therefore, the measure~$\rho$ can be split as~$\rho(x, y) = \rho(y|x)\nu(x)$. 

For the probability measure~$\rho$ on~$X\times Y$, we assume that
\begin{equation}\label{Y.leq.M.1}
\int_Z\norm{y}_Y^2~d\rho(x,y)<\infty.
\end{equation}

For the considered model~$y=\gp(x)+\varepsilon$ with centred noise~$\varepsilon$ we find~$\int_Y y d\rho(y|x)= \gp(x)$ provided that the conditional expectation w.r.t.~$\rho$ of~$y$ given~$x$ exists (a.s.). This holds true under Assumption~\eqref{Y.leq.M.1}. This fact together with the operator equation~\eqref{Op.eqn} motivates us to consider the following assumption. 
\begin{assumption}[The true solution]\label{Ass:fp}
The conditional expectation w.r.t.~$\rho$ of~$y$ given~$x$ exists (a.s.), and there exists unique~$\fp \in \mathrm{int}(\mathcal{D}(A))\subset\HH~$ such that
  \begin{equation*}
\int_Y y d\rho(y|x)= A(\fp)(x), \text{ for all } x\in X.
  \end{equation*}
\end{assumption}
Here,~$\fp$ is the true solution of equation~\eqref{Op.eqn} which we aim at estimating. Here, we want to mention that the function~$\fp$ is also the minimizer of the expected risk considered in~\cite{Rastogi2020a}. 

We consider a \emph{Bernstein-type assumption} for the noise~$\varepsilon=y-A(\fp)(x)$:

\begin{assumption}[Noise condition]\label{Ass:noise}
There exist some constants~$M,\Sigma$ such that for almost all~$x\in X$,
\begin{equation*}
\int_Y\left(e^{\norm{\varepsilon}_Y/M}-\frac{\norm{\varepsilon}_Y}{M}-1\right)d\rho(y|x)\leq\frac{\Sigma^2}{2M^2}.
\end{equation*}
\end{assumption}

We want to utilize the properties of the Reproducing kernel Hilbert spaces (RKHSs) in our analysis. Therefore, we assume that~$\operatorname{Ran}(A)$ is contained in a vector-valued Reproducing kernel Hilbert space (RKHSvv).   The RKHSvv~$\HH_K$ arises from the operator-valued positive semi-definite kernel~$K:X\times X\to \mathcal{L}(Y)$~\cite{Micchelli1}. Here,~$\mathcal{L}(Y)$ is the Banach space of bounded linear operators.

\begin{assumption}[Vector valued reproducing kernel Hilbert space~$\HH'$] \label{Ass:kernel}
Suppose~$\HH'$ is an RKHSvv of functions~$g:X\to Y$ corresponding to the kernel~$K:X\times X\to \mathcal{L}(Y)$ such that
  \begin{enumerate}[(i)]
  \item For all~$x\in X$,~$K_x:Y\to\HH'$ is a Hilbert-Schmidt
    operator, and
    \[\kappa^2:=\sup_{x \in X} \norm{K_x}^2_{HS} = {\sup_{x \in
          X}\tr(K_x^*K_x)}<\infty.\] 
  \item The real-valued function~$\varsigma:X\times X \to \RR$, defined by~$\varsigma(x,t)=\inner{ K_tv,K_xw}_{\HH'}$, is measurable~$\forall v,w\in Y$.
  \end{enumerate}
\end{assumption}

This assumption implies that~$\HH'\subset \LL$. We denote the canonical injection map~$\HH'$ to~$\LL$ by~$\ip$ and the corresponding covariance operator is~$\tp:= \ip^{\ast}\ip$. From the above assumption, we see that the covariance operator is positive and trace class. The covariance operator is very important in our convergence analysis. We will need some regularity assumptions in terms of the covariance operator on the marginal probability measure~$\nu$ to achieve the uniform convergence rates for the regularized solution~\eqref{fzl}. 

The error estimates studied in our analysis are based on the smoothness of the true solution and the behaviour of the effective dimension. The error estimates and the optimal parameter choice depend on the effective dimension for the regularization methods in reproducing kernel Hilbert spaces~\cite{Caponnetto,Blanchard,Rastogi2020}. To achieve the fast convergence rates, we introduce the concept of the effective dimension~$\mathcal{N}(\la)$~\cite{Zhang}:
$$\mathcal{N}(\la):=Tr\left((\tp+\la I)^{-1}\tp\right), \text{  for }\la>0.$$

The effective dimension is a continuous, decreasing function of~$\la$. The effective dimension is finite, since the operator~$\tp$ is a trace class, and we get
$$
\mathcal{N}(\la)\leq \norm{(\tp+\la I)^{-1}}_{\mathcal{L}(\HH)}Tr\left(\tp\right) \leq \frac{\kappa^2}{\la}.
$$
%Finally, the spectral decomposition of~$\tp$ gives
%\begin{equation}\label{nl}
%  \mathcal{N}(\la)\geq \frac{\norm{\tp}_{\lh}}{\norm{\tp}_{\lh}+\la}\geq \frac{1}{2}\qquad \text{for}\quad \la \leq \norm{\tp}_{\lh}.
%\end{equation}

The different behaviours of the eigenvalues of the covariance operator lead to different decay rates of the effective dimension~\cite{Lu2020}. Under the different scenarios of the effective dimension, we will get the explicit convergence rates in the next section.

In order to establish the error estimate, we introduce the discrete operators for the samples. For the ordered set~$(\xx)_i=x_i$, we define the {\it Sampling Operator }
$$(\sx(g))_i=g(x_i)\quad \text{and} \qquad 1\leq i \leq m.$$

We define the inner product space~$Y^m$ with the inner product~$\inner{\yy,\yy'}_{m}=\frac{1}{m}\inner{y_i,y_i'}_{Y}$ for~$(\yy)_i=y_i$ and~$(\yy')_i=y_i'$ for~$1\leq i \leq m$.
Then, we get the expression of its adjoint~$\sx^*$ as
$$\sx^*\yy=\frac{1}{m}\sum_{i=1}^m K_{x_i} y_i,~~~~\forall \yy\in Y^m.$$

It can be easily checked that under Assumption~\ref{Ass:kernel},~$\norm{\sx}_{\HH'\to Y^m}\leq \kappa$.

We need to make some assumptions about the nonlinear structure of operator~$A$. Following the work of Werner and Hofmann~\cite{Werner2019a}, we consider the following assumption on~$A$,~$\da$, and~$\fp$. To introduce this assumption, we define the closed balls~$B_\mu^u(\fp)=\brac{f\in\HH_u:\norm{f-\fp}_u\leq \mu}$  in~$\HH_u~(u \in \RR)$ with center~$\fp \in \HH_u$ and radius~$\mu$~$(0 < \mu\leq 1)$ and  their intersections with the domain of~$A$,~$\mathcal{D}_\mu^u(\fp):=B_\mu^u(\fp)\cap\da$. For the simplicity, we will denote~$B_\mu^0(\fp)$ and~$\mathcal{D}_\mu^0(\fp)$ and by~$B_\mu(\fp)$ and~$\mathcal{D}_\mu(\fp)$.

\begin{assumption}\label{Ass:A}
\begin{enumerate}[(i)]

\item The domain~$\da$ of~$A$ is a convex and closed subset of~$\HH$.

\item The operator~$A : \da \to \HH'$ is weak-to-weak sequentially continuous\footnote{i.e.,~$f_n \rightharpoonup \hat{f}\in\HH$ with~$f_n \in \da$,~$n \in \NN$, and~$\hat{f} \in \da$ implies~$A(f_n)\rightharpoonup  A(\hat{f}) \in \HH'$.}.

\item The operator~$A$ is Lipschitz continuous with Lipschitz constant~$\ella < \infty$ in a sufficiently large ball~$\mathcal{B}_d(\fp)$,
\begin{equation*}\label{A.cont}
\norm{A(f)-A(\tilde{f})}_{\HH'}\leq \ella \norm{f-\tilde{f}}_{\HH} \qquad\forall f ,\tilde{f}\in \mathcal{B}_d(\fp) \cap \mathcal{D}(A) \subset \HH,
\end{equation*}

\item There exist constants~$p\geq0$,~$s> 0$,~$\al> 0$,~$d > 0$,~$\theta\geq 0$ and~$Q \subset\mathcal{D}_d^\theta(\fp)\cap \da$ such that
\begin{equation*}
\norm{f-\fp}_{\HH_{-p}}\leq \al\norm{\ip\sbrac{A(f)-A(\fp)}}_{\LL}^s
\end{equation*}
holds for all~$f \in Q$, where the constant~$\al$ may depend on~$p$,~$s$, and~$Q$.
\end{enumerate}
\end{assumption}

Assumption~\ref{Ass:A}~(iv) is called the conditional stability estimate which helps us to characterize the degree of ill-posedness of inverse problems. Here, we note that operator~$A$ may not be differentiable, (see the examples in~\cite{Werner2019a}).

\section{Convergence analysis}\label{Sec:Analysis}
The assertions about the convergence of Tikhonov-regularized solution~$\fz$ to the true solution~$\fp$ are formulated in this section. First of all, we introduce some standard quantities required to establish the error estimates. We denote
\begin{align}
\Theta_{\zz}:=&\norm{(\tp +\la I)^{-1/2}\sx^*\bm{\varepsilon}}_{\HH'}\qquad \text{for} \quad \bm{\varepsilon}=\sx \sbrac{A(\fp )}-\yy, \label{theta.z} \\
%\intertext{and}
\Psi_\xx:=&\norm{(\tp +\la I)^{-1/2}(\tp-\tx)}_{\mathcal{L}_2(\HH')}.  \label{psi}
%\Gamma_{\xx}:=&\norm{(\tx+\la I)^{-1/2}(\tp +\la I)^{1/2}}_{\mathcal{L}(\HH')}. \label{gamma}
\end{align}
%(\mathbb{\varepsilon})_i=A(\fp)(x_i)-y_i \quad \text{for} \quad 1\leq i \leq  m
 
The probabilistic estimates of the above quantities are given in Appendix~\ref{Sec:prob.est}. We will use the following standard assumption on the sample size~$m$ and the regularization parameter~$\la$ for our probabilistic estimates:
\begin{equation}\label{l.la.condition}
\mathcal{N}(\la) \leq  m\la \qquad \text{and}\qquad 0<\la\leq 1. 
\end{equation}

Now, we introduce the concept of the distance function (also known as `approximate source conditions') which can be used in the absence of the source condition for~$\fp$~\cite{Baumeister1987,Smale2003}. It measures the violation of a benchmark smoothness of the true solution. It becomes very important in the `oversmoothing case'~$\fp \notin\DD(L)$ for regularization in Hilbert Scales. 

\begin{definition}[Approximate source condition]
For given~$q$, we define the distance function~$d : [0, \infty)\to[0, \infty)$ by
\begin{align}\label{Defi:dist}
\dr=\inf\brac{\norm{f-\fp}_{\HH}:f-\fbar= L^{-q}v \text{ and }\norm{v}_{\HH} \leq R},\quad R>0.
\end{align}
\end{definition}
Here~$q$ defines the benchmark smoothness. Let~$\fr$ be the minimizing element of the above problem. Here, we also denote the quantities~$\dra=\norm{\ip\sbrac{A(\fr)-A(\fp)}}_{\LL}$ and~$\drp=\norm{\fr-\fp}_{\HH_{-p}}$.

Here, we note that when the true solution is of the form~$\fp=L^{-q}u$ with~$\norm{u}_{\HH} \leq \bar{R}$, then the distance function~$\DD(\bar{R})=0$ and the minimizer~$\fp^{\bar{R}}=\fp$.

The error analysis starts using the fact that~$\fz$ is the minimizer of the Tikhonov functional~\eqref{fzl}. We get the deterministic expressions~\eqref{err_1.2},~\eqref{p_s.2} for the quantities~$\norm{\ip\sbrac{A(\fz)-A(\fp)}}_{\LL}$ and~$\norm{L(\fz-\fp)}_{\HH}$ after some rearrangement, and using Cauchy-Schwarz inequality, Young's inequality. After the simplification and using the probabilistic estimates from Proposition~\ref{main.bound} we get the error estimates in terms of the sample size~$m$, the regularization parameter~$\la$, and distance function by $R(\la)$. The distance function can be measured using the source condition for~$\fp$ (see Section~\ref{Sec:Explicit.rates}). Consequently, we get the explicit dependency~$\la \to R(\la)$. The estimates depend on the effective dimension which will be explicitly expressed in terms of~$\la$ by using the different decay conditions on the effective dimension. Then, the bounds can be expressed explicitly in terms of~$\la$ and~$m$ for the given smoothness of the solution~$\fp$. In Section~\ref{Sec:Explicit.rates}, the a-priori choice of regularization parameter will be obtained by balancing the terms in the error bounds. 

%The above error estimates depends on the distance function~$\dra$. The distance function can be estimated explicitly under the source condition.

\begin{theorem}\label{err.upper.bound.p.1}
Let Assumptions~\ref{Ass:fp}--\ref{Ass:A}, and condition~\eqref{l.la.condition} hold true. Let~$1\leq q \leq 2 + p$,~$q(s-1)\leq p+s$ and~$\fz,\fr\in Q$ (for sufficiently large sample size~$m$) for~$p$,~$s$,~$Q$,~$q$ defined in Assumption~\ref{Ass:A}~(iv),~\eqref{Defi:dist}. Then, for all~$0<\eta<1$, the following  bounds hold with the confidence~$1-\eta$:
\begin{align*}
\norm{\ip\sbrac{A(\fz)-A(\fr)}}_{\LL} \leq \widetilde{C}\la^{\frac{1}{2}}&\brac{\sqrt{\frac{\mathcal{N}(\la)}{m\la}}+R(\la)^{\frac{p+1}{(p+q)-s(q-1)}}\la^{\frac{s(q-1)}{2(p+q)-2s(q-1)}}}\log\paren{\frac{4}{\eta}},\\
\norm{L(\fz-\fr)}_{\HH}\leq \widetilde{C}&\brac{\sqrt{\frac{\mathcal{N}(\la)}{m\la}}+R(\la)^{\frac{p+1}{(p+q)-s(q-1)}}\la^{\frac{s(q-1)}{2(p+q)-2s(q-1)}}}\log\paren{\frac{4}{\eta}},\\
\norm{\fz-\fr}_{\HH}\leq \widetilde{C}\la^{\frac{s}{2(p+1)}}&\brac{\sqrt{\frac{\mathcal{N}(\la)}{m\la}}+R(\la)^{\frac{p+1}{(p+q)-s(q-1)}}\la^{\frac{s(q-1)}{2(p+q)-2s(q-1)}}}^{\frac{(p+s)}{(p+1)}}\log\paren{\frac{4}{\eta}}.
\end{align*}
Here,~$R(\la)$ is the solution of the equation~$\dra R^{-\frac{p+1}{(p+q)-s(q-1)}}= \la^{\frac{p+q}{2(p+q)-2s(q-1)}}$ for~$\dra \neq 0$ and~$R(\la)$ is a fixed constant for~$\dra =0$.
\end{theorem}

\begin{proof}
By the definition of~$\fz$ as the solution to the minimization problem in~\eqref{fzl}, we have
\begin{equation*}
\frac{1}{m}\sum\limits_{i=1}^m\norm{\sbrac{A(\fz)}(x_i)-y_i}_Y^2+\la \norm{L(\fz-\fbar)}_{\HH}^2\leq \frac{1}{m}\sum\limits_{i=1}^m\norm{\sbrac{A(\fr)}(x_i)-y_i}_Y^2+\la \norm{L(\fr-\fbar)}_{\HH}^2.
\end{equation*}

We re-express the above inequality as follows,
\begin{equation*}\label{idea.1}
\norm{\sx \sbrac{A(\fz)}-\yy}_m^2+\la \norm{L(\fz-\fbar)}_{\HH}^2\leq \norm{\sx \sbrac{A(\fr)}-\yy}_m^2+\la \norm{L(\fr-\fbar)}_{\HH}^2
\end{equation*}
which implies
\begin{align*}
&\norm{\sx\sbrac{A(\fz)-A(\fr)}}_m^2+2\inner{\sx\sbrac{A(\fz)-A(\fr)}, \sx \sbrac{A(\fr )}-\yy}_m+\la\norm{L(\fz-\fr)}_{\HH}^2   \\  \nonumber
\leq& 2\la\inner{ L(\fz-\fr),L(\fbar-\fr)}_{\HH}.
\end{align*}

Then we have,
\begin{align}\label{B1.1}
&\norm{\ip\sbrac{A(\fz)-A(\fr)}}_{\LL}^2+\la\norm{L(\fz-\fr)}_{\HH}^2  \\ \nonumber
\leq  & 2\la\inner{ L(\fz-\fr),L(\fbar-\fr)}_{\HH}+2\inner{\ip\sbrac{A(\fz)-A(\fr)},\ip\sbrac{A(\fr)-A(\fp)}}_{\HH'}\\   \nonumber
&+2\inner{A(\fz)-A(\fr),(\tp-\tx)\sbrac{A(\fr)-A(\fp)}+\sx^*\bm{\varepsilon}}_{\HH'}  \\  \nonumber
&+\inner{A(\fz)-A(\fr),(\tp-\tx)\sbrac{A(\fz)-A(\fr)}}_{\HH'}.
\end{align}

Using the interpolation inequality~\eqref{interpolation}, the definition of distance function~\eqref{Defi:dist} and for~$\fz \in Q$ under Assumption~\ref{Ass:A}, we obtain
\begin{align}\label{B_2.1}
\inner{ L(\fz-\fr),L(\fbar-\fr)}_{\HH}\leq & \norm{\fr-\fbar}_{\HH_{q}}\norm{\fz-\fr}_{\HH_{2-q}}\\  \nonumber
\leq & R\norm{L(\fz-\fr)}_{\HH}^{\frac{p-q+2}{p+1}}\norm{\fz-\fr}_{\HH_{-p}}^{\frac{q-1}{p+1}}.
%\leq & C\norm{L(\fz-\fr)}_{\HH}^{\frac{p-q+2}{p+1}}\norm{\ip\sbrac{A(\fz)-A(\fr)}}_{\LL}^{\frac{s(q-1)}{p+1}},
\end{align}
%where~$C=\al^{\frac{q-1}{p+1}}\norm{\fr-\fbar}_{\HH_{q}}$.

We have
\begin{align*}
\norm{\fz-\fr}_{\HH_{-p}} \leq &\norm{\fz-\fp}_{\HH_{-p}}+\norm{\fp-\fr}_{\HH_{-p}}\\
\leq &\al\norm{\ip\paren{A(\fz)-A(\fp)}}_{\LL}^s+\al\norm{\ip\paren{A(\fp)-A(\fr)}}_{\LL}^s\\
\leq & \al\norm{\ip\paren{A(\fz)-A(\fp)}}_{\LL}^s+2\al\norm{\ip\paren{A(\fp)-A(\fr)}}_{\LL}^s
\end{align*}
which implies
\begin{align*}
\norm{\fz-\fr}_{\HH_{-p}}^{\frac{2(q-1)}{(p+q)}} \leq & 2 \al^{\frac{2(q-1)}{(p+q)}}\norm{\ip\paren{A(\fz)-A(\fp)}}_{\LL}^{\frac{2s(q-1)}{(p+q)}}+ 4 \al^{\frac{2(q-1)}{(p+q)}}\dra^{\frac{2s(q-1)}{(p+q)}},
\end{align*}
where~$\dra=\norm{\ip\sbrac{A(\fr)-A(\fp)}}_{\LL}$.

Now we apply Young's inequality ($ab\leq \frac{a^u}{u}+\frac{b^v}{v}$ for~$\frac{1}{u}+\frac{1}{v}=1$) with~$a=\paren{\frac{u}{4}}^{\frac{1}{u}}\norm{L(\fz-\fr)}_{\HH}^{\frac{2}{u}}$,~$b=\paren{\frac{4}{u}}^{\frac{1}{u}}R\norm{\fz-\fr}_{\HH_{-p}}^{\frac{q-1}{p+1}}$,~$u=\frac{2p+2}{p-q+2}$ and~$v = \frac{2p+2}{p+q}$ in~\eqref{B_2.1}, and this implies
\begin{align*}
\inner{ L(\fr-\fz),L(\fr-\fbar)}_{\HH}\leq &\frac{1}{4}\norm{L(\fz-\fr)}_{\HH}^2+CR^{\frac{2p+2}{p+q}}\norm{\fz-\fr}_{\HH_{-p}}^{\frac{2(q-1)}{p+q}},
\end{align*}
where~$C=\frac{1}{v}\paren{\frac{4}{u}}^{\frac{v}{u}}$. Now, using~\eqref{Ass:A} we get,
\begin{align}\label{B2.1}
&\inner{ L(\fr-\fz),L(\fr-\fbar)}_{\HH}\\   \nonumber
\leq &\frac{1}{4}\norm{L(\fz-\fr)}_{\HH}^2+C'^2 R^{\frac{2p+2}{p+q}}\norm{\ip\paren{A(\fz)-A(\fp)}}_{\LL}^{\frac{2s(q-1)}{(p+q)}}+2C'^2R^{\frac{2p+2}{p+q}}\dra^{\frac{2s(q-1)}{(p+q)}},
\end{align}
where~$C'^2=2C\al^{\frac{2(q-1)}{(p+q)}}$.

To estimate the last two terms in~\eqref{B1.1} we consider the inequality
\begin{align*}
\inner{f,g}_{\HH'}= & \la\inner{f,(\tp+\la I)^{-1}g}_{\HH'}+\inner{f,\tp(\tp+\la I)^{-1}g}_{\HH'}\\
\leq & \brac{\sqrt{\la}\norm{f}_{\HH'}+\norm{\ip f}_{\LL}}\norm{(\tp+\la I)^{-1/2}g}_{\HH'}.
\end{align*}
By taking~$f=A(\fz)-A(\fr)$, and~$g=(\tp-\tx)\sbrac{A(\fz)-2 A(\fp)+A(\fr)}+2\sx^*\bm{\varepsilon}$ and using~\eqref{L.unbound},~Assumption~\ref{Ass:A}~(iii) we get,
\begin{align}\label{B3.1}
&\inner{A(\fz)-A(\fr),(\tp-\tx)\sbrac{A(\fz)-2 A(\fp)+A(\fr)}+2\sx^*\bm{\varepsilon}}_{\HH'}\\  \nonumber
\leq & \brac{\ell_1\Psi_\xx+2\Theta_{\zz}}\brac{\sqrt{\la}\norm{A(\fz)-A(\fr)}_{\HH'}+\norm{\ip\sbrac{A(\fz)-A(\fr)}}_{\LL}}\\   \nonumber
\leq &\brac{\ell_1\Psi_\xx+2\Theta_{\zz}}\brac{\ella\sqrt{\la}\norm{\fz-\fr}_{\HH}+\norm{\ip\sbrac{A(\fz)-A(\fr)}}_{\LL}}\\   \nonumber
\leq &\brac{\ell_1\Psi_\xx+2\Theta_{\zz}}\brac{\ell\sqrt{\la}\norm{L(\fz-\fr)}_{\HH}+\norm{\ip\sbrac{A(\fz)-A(\fr)}}_{\LL}},
\end{align}
where~$\ell_1=\norm{A(\fz)-2A(\fp)+A(\fr)}_{\HH'}$,~$\ell=\frac{\ella}{\elll}$ and~$\ella$,~$\elll$,~$\Theta_\zz$,~$\Psi_\xx$ are defined in~Assumption~\ref{Ass:A}~(iii),~\eqref{L.unbound},~Assumption~\ref{Ass:A}~(iii),~\eqref{theta.z},~\eqref{psi}, respectively.

Using the above estimates~\eqref{B2.1},~\eqref{B3.1} in~\eqref{B1.1} we obtain,
\begin{align*}
&\norm{\ip\sbrac{A(\fz)-A(\fr)}}_{\LL}^2+\frac{\la}{2}\norm{L(\fz-\fr)}_{\HH}^2  \\
\leq &2C'^2\la R^{\frac{2p+2}{p+q}}\norm{\ip\sbrac{A(\fz)-A(\fr)}}_{\LL}^{\frac{2s(q-1)}{p+q}}+4 C'^2\la R^{\frac{2p+2}{p+q}}\dra^{\frac{2s(q-1)}{(p+q)}}\\
&+2\norm{\ip\sbrac{A(\fr)-A(\fp)}}_{\LL}\norm{\ip\sbrac{A(\fz)-A(\fr)}}_{\LL}\\
&+(\ell_1\Psi_\xx+2\Theta_{\zz})\norm{\ip\sbrac{A(\fz)-A(\fr)}}_{\LL}\\
&+\ell\sqrt{\la}(\ell_1\Psi_\xx+2\Theta_{\zz})\norm{L(\fz-\fr)}_{\HH}.
\end{align*}

Now, using the inequality~$ab\leq a^2+b^2$ we get,
\begin{align*}
&\norm{\ip\sbrac{A(\fz)-A(\fr)}}_{\LL}^2+\frac{\la}{2}\norm{L(\fz-\fr)}_{\HH}^2  \\
\leq &2C'^2\la R^{\frac{2p+2}{p+q}}\norm{\ip\sbrac{A(\fz)-A(\fr)}}_{\LL}^{\frac{2s(q-1)}{p+q}}+4 C'^2\la R^{\frac{2p+2}{p+q}}\dra^{\frac{2s(q-1)}{(p+q)}}\\
&+4\norm{\ip\sbrac{A(\fr)-A(\fp)}}_{\LL}^2+\frac{1}{4}\norm{\ip\sbrac{A(\fz)-A(\fr)}}_{\LL}^2\\
&+(\ell_1\Psi_\xx+2\Theta_{\zz})^2+\frac{1}{4}\norm{\ip\sbrac{A(\fz)-A(\fr)}}_{\LL}^2\\
&+\ell^2(\ell_1\Psi_\xx+2\Theta_{\zz})^2+\frac{\la}{4}\norm{L(\fz-\fr)}_{\HH}.
\end{align*}
which implies
\begin{align*}
&\frac{1}{2}\norm{\ip\sbrac{A(\fz)-A(\fr)}}_{\LL}^2+\frac{\la}{4}\norm{L(\fz-\fr)}_{\HH}^2  \\
\leq &2C'^2\la R^{\frac{2p+2}{p+q}}\norm{\ip\sbrac{A(\fz)-A(\fr)}}_{\LL}^{\frac{2s(q-1)}{p+q}}+4 C'^2\la R^{\frac{2p+2}{p+q}}\dra^{\frac{2s(q-1)}{(p+q)}}\\
&+4\dra^2+(\ell^2+1)(\ell_1\Psi_\xx+2\Theta_{\zz})^2.
\end{align*}
Now by rearranging the terms we obtain,
\begin{align*}
&\norm{\ip\sbrac{A(\fz)-A(\fr)}}_{\LL}^2+\la\norm{L(\fz-\fr)}_{\HH}^2  \\
\leq &\paren{8C'^2\la R^{\frac{2p+2}{p+q}}\norm{\ip\sbrac{A(\fz)-A(\fr)}}_{\LL}^{\frac{2s(q-1)}{p+q}}-\norm{\ip\sbrac{A(\fz)-A(\fr)}}_{\LL}^2}\\
&+16 C'^2\la R^{\frac{2p+2}{p+q}}\dra^{\frac{2s(q-1)}{(p+q)}}+16\dra^2+4(\ell^2+1)(\ell_1\Psi_\xx+2\Theta_{\zz})^2\\
\leq & \sup_{\tau\geq 0}\paren{8C'^2\la R^{\frac{2p+2}{p+q}}\tau^{\frac{2s(q-1)}{p+q}}-\tau^2}+16 C'^2\la R^{\frac{2p+2}{p+q}}\dra^{\frac{2s(q-1)}{(p+q)}}\\
&+16\dra^2+4(\ell^2+1)(\ell_1\Psi_\xx+2\Theta_{\zz})^2 \\
= & C''^2R^{\frac{2p+2}{(p+q)-s(q-1)}}\la^{\frac{p+q}{(p+q)-s(q-1)}}+16 C'^2\la R^{\frac{2p+2}{p+q}}\dra^{\frac{2s(q-1)}{(p+q)}}\\
&+16\dra^2+4(\ell^2+1)(\ell_1\Psi_\xx+2\Theta_{\zz})^2,
\end{align*}
where~$C''=\paren{\frac{C'^2 8s(q-1)}{p+q}}^{\frac{(p+q)}{2(p+q)-2s(q-1)}}\paren{\frac{(p+q)-s(q-1)}{s(q-1)}}^{1/2}$.

Hence we get,
\begin{align}\label{err_1.1}
&\norm{\ip\sbrac{A(\fz)-A(\fr)}}_{\LL}\\  \nonumber
\leq &2(\ell+1)(\ell_1\Psi_\xx+2\Theta_{\zz})+C''R^{\frac{p+1}{(p+q)-s(q-1)}}\la^{\frac{p+q}{2(p+q)-2s(q-1)}}+4 C'\sqrt{\la}R^{\frac{p+1}{p+q}}\dra^{\frac{s(q-1)}{(p+q)}}+4\dra
\end{align}
and
\begin{align}\label{p_s.1}
&\norm{L(\fz-\fr)}_{\HH}\\   \nonumber
\leq &\frac{1}{\sqrt{\la}}\brac{2(\ell+1)(\ell_1\Psi_\xx+2\Theta_{\zz})+C''R^{\frac{p+1}{(p+q)-s(q-1)}}\la^{\frac{p+q}{2(p+q)-2s(q-1)}}+4 C'\sqrt{\la}R^{\frac{p+1}{p+q}}\dra^{\frac{s(q-1)}{(p+q)}}+4\dra}.
\end{align}

In case~$\dra=0$, for some fixed~$\bar{R}$, we get explicit bounds from~\eqref{err_1.1} and~\eqref{p_s.1} in terms of~$m$ and~$\la$ using~\eqref{Theta.bound} and~\eqref{Psi.bound}.

For~$\dra\neq 0$ and~$\la>0$, we optimize the bounds by balancing the terms in~$R$ and~$\la$. Let~$R = R(\lambda)$ solves the equation~$\Gamma(R) := \dra R^{-\frac{p+1}{(p+q)-s(q-1)}}= \la^{\frac{p+q}{2(p+q)-2s(q-1)}}$. The function~$\Gamma(R)$ is a non-vanishing decreasing function, and hence the inverse~$\Gamma^{-1}$ exists, and it is decreasing. With this, the error bounds can be expressed as
\begin{align}\label{err_1.2}
&\norm{\ip\sbrac{A(\fz)-A(\fr)}}_{\LL}\\  \nonumber
\leq &2(\ell+1)(\ell_1\Psi_\xx+2\Theta_{\zz})+C'''R(\la)^{\frac{p+1}{(p+q)-s(q-1)}}\la^{\frac{p+q}{2(p+q)-2s(q-1)}}
\end{align}
and
\begin{align}\label{p_s.2}
&\norm{L(\fz-\fr)}_{\HH}\\   \nonumber
\leq &\frac{1}{\sqrt{\la}}\brac{2(\ell+1)(\ell_1\Psi_\xx+2\Theta_{\zz})+C'''R(\la)^{\frac{p+1}{(p+q)-s(q-1)}}\la^{\frac{p+q}{2(p+q)-2s(q-1)}}}.
\end{align}
where~$C'''=C''+4 C'+4$.

Now, using~\eqref{Theta.bound} and~\eqref{Psi.bound} in~\eqref{err_1.1},~\eqref{p_s.1} we obtain with probability~$1-\eta$,
\begin{equation}\label{err1.3}
\norm{\ip\sbrac{A(\fz)-A(\fr)}}_{\LL}\leq \widetilde{C}\brac{\sqrt{\frac{\mathcal{N}(\la)}{m}}+R(\la)^{\frac{p+1}{(p+q)-s(q-1)}}\la^{\frac{p+q}{2(p+q)-2s(q-1)}}}\log\paren{\frac{4}{\eta}}
\end{equation}
and
\begin{equation}\label{ps.3}
\norm{L(\fz-\fr)}_{\HH}\leq \frac{\widetilde{C}}{\sqrt{\la}}\brac{\sqrt{\frac{\mathcal{N}(\la)}{m}}+R(\la)^{\frac{p+1}{(p+q)-s(q-1)}}\la^{\frac{p+q}{2(p+q)-2s(q-1)}}}\log\paren{\frac{4}{\eta}},
\end{equation}
where~$\widetilde{C}$ depends on~$\ell$,~$\ell_1$,~$p$,~$q$,~$s$,~$\kappa$,~$M$,~$\Sigma$,~$\al$.

Taking the mean using the inequality~\eqref{interpolation} we get,
\begin{align*}
&\norm{\fz-\fr}_{\HH}\\
\leq & \norm{L(\fz-\fr)}^{\frac{p}{p+1}}_{\HH}\norm{\fz-\fr}^{\frac{1}{p+1}}_{\HH_{-p}}\\
\leq & \al^{\frac{1}{p+1}}\norm{L(\fz-\fr)}_{\HH}^{\frac{p}{p+1}}\norm{\ip\sbrac{A(\fz)-A(\fr)}}_{\LL}^{\frac{s}{p+1}}\\
\leq & \widetilde{C}\la^{-\frac{p}{2(p+1)}}\brac{\sqrt{\frac{\mathcal{N}(\la)}{m}}+R(\la)^{\frac{p+1}{(p+q)-s(q-1)}}\la^{\frac{p+q}{2(p+q)-2s(q-1)}}}^{\frac{(p+s)}{(p+1)}}\log\paren{\frac{4}{\eta}}.
\end{align*}

Hence the proof completes.
\end{proof}

\section{Explicit rates under source condition}\label{Sec:Explicit.rates}

Here, we consider the smoothness of~$\fp$ by the source condition in terms of the operator~$L^{-1}$ to get the explicit rates in terms of~$m$ and~$\la$. The smoothness parameter~$r$ influences the rates of convergence, the larger~$r$ (Smoother~$\fp$) will lead to the faster convergence rates. 

%But the rates will not be improved beyond~$q=2+p$ due to the well-known saturation effect of the Tikhonov regularization. 
%Here, we consider that the true solution have smoothness in terms of the operator~$L^{-1}$. 

\begin{assumption}[General source condition]\label{source.cond}
The true solution~$\fp$ satisfy the condition:
  \begin{equation*}
    \fp-\fbar= L^{-r}v \text{ and }\norm{v}_{\HH} \leq R^\dagger.
  \end{equation*}
\end{assumption}
%The index function~$\theta$ is positive, continuous, and increasing with~$\theta(0)=0$. In particular,~$\theta$ can be power-type~$\theta(t)=t^r$ in H{\"o}lder source condition or logarithm-type~$\theta(t)=t^p\log^{-\nu}\paren{\frac{1}{t}}$ with~$p\in\NN$,~$\nu\in [0, 1]$. 

The rates in Theorem~\ref{err.upper.bound.p.1} can be further simplified in two cases based on the behaviour of the distance function~$\dra$. 

%\subsection{Distance function~$\dra=0$}
In case~$\dra=0$, we get the explicit error bounds in terms of~$\la$ and~$m$ from Theorem~\ref{err.upper.bound.p.1}. We get the function~$d_A(\bar{R})=0$ when~$\fp-\fbar=L^{-q}v$ and~$\norm{v}\leq \bar{R}$ for some~$\bar{R}$, i.e.,~$r \geq q$. Consequently, this also implies~$\fp^{\bar{R}}=\fp$. So, the rates of convergence in the reconstruction norm and prediction norm can be given as:   
\begin{align*}
\PP \Big\{\norm{\ip\sbrac{A(\fz)-A(\fp)}}_{\mathcal{L}^2} \leq \widetilde{C}\la^{\frac{1}{2}}&\paren{\sqrt{\frac{\mathcal{N}(\la)}{m\la}}+\bar{R}^{\frac{p+1}{(p+q)-s(q-1)}}\la^{\frac{s(q-1)}{2(p+q)-2s(q-1)}}}\log\paren{\frac{4}{\eta}}\Big\}\geq 1-\eta,\\
\PP \Big\{\norm{\fz-\fp}_{\HH}\leq \widetilde{C}\la^{\frac{s}{2(p+1)}}&\paren{\sqrt{\frac{\mathcal{N}(\la)}{m\la}}+\bar{R}^{\frac{p+1}{(p+q)-s(q-1)}}\la^{\frac{s(q-1)}{2(p+q)-2s(q-1)}}}^{\frac{(p+s)}{(p+1)}}\log\paren{\frac{4}{\eta}}\Big\}\geq 1-\eta.
\end{align*}

By balancing the error terms, we choose the regularization parameter~$\la$ in terms of the sample size~$m$. Consequently, we get the explicit rates of convergence in terms of the sample size. 
\begin{corollary}\label{cor.err.upper.bound.gen}
Under the same assumptions of Theorem~\ref{err.upper.bound.p.1} and Assumption~\ref{source.cond} with~$r\geq q$ and the a-priori choice of the regularization parameter~$\la^*=\Theta_{\mathcal{N},u}^{-1}\paren{\frac{1}{\sqrt{m}}}$ for~$\Theta_{\mathcal{N},u}(t)=\frac{t^u}{\sqrt{\mathcal{N}(t)}}$ and~$u=\frac{p+q}{2(p+q)-2s(q-1)}$, for all~$0<\eta<1$, the following error estimates holds with confidence~$1-\eta$:
\begin{align*}
\norm{\ip\sbrac{A(\fz)-A(\fp)}}_{\mathcal{L}} \leq \overline{C}\paren{\la^*}^u\log\paren{\frac{4}{\eta}}
\end{align*}
and
\begin{align*}
\norm{\fz-\fp}_{\HH}\leq &  \overline{C}\paren{\la^*}^{\frac{2u(s+p)-p}{2(p+1)}}\log\paren{\frac{4}{\eta}}
=   \overline{C}\paren{\la^*}^{\frac{sq}{2(p+q)-2s(q-1)}}\log\paren{\frac{4}{\eta}}.
\end{align*}
where~$\overline{C}$ depends on~$\ell$,~$\ell_1$,~$p$,~$q$,~$s$,~$\kappa$,~$M$,~$\Sigma$,~$\al$,~$\bar{R}$.
\end{corollary}

In case~$\dra\neq 0$, we have to estimate the function~$\dra$ explicitly. We utilize the result of~\cite[Theorem~5.9]{Hofmann2007} to estimate the distance function using the source condition. For the benchmark smoothness~$q$ and the given smoothness~$r$, we assume that~$q\geq r$ and~$2q\geq p+r$. Then, under Assumption~\ref{source.cond}, we get the bound
$$
d(R) \leq \frac{\paren{R^\dagger}^{\frac{q}{q-r}}}{R^{\frac{r}{q-r}}},\quad R>0.
$$

Following the analysis in~\cite[Theorem~5.9]{Hofmann2007} we also obtain the bounds for the distance function:
\begin{equation}\label{R.choice}
  \drp:=\norm{\fr-\fp}_{\HH_{-p}}  \leq \frac{\paren{R^\dagger}^{\frac{q+p}{q-r}}}{R^{\frac{r+p}{q-r}}},\quad R>0.
\end{equation}

To bound the distance function~$\dra$, we assume the following assumption in addition to Assumption~\ref{Ass:A}~(iv) with the same parameters:
\begin{assumption}\label{Cond.est}
\item There exists a constant~$\beta>0$ such that
\begin{equation*}
 \al\norm{\ip\sbrac{A(f)-A(\fp)}}_{\LL}^s \leq \beta\norm{f-\fp}_{\HH_{-p}}
\end{equation*}
holds for all~$f \in Q$.
\end{assumption}

Now, according to Theorem~\ref{err.upper.bound.p.1} we have to solve the following equation in order to estimate~$R$ in terms of~$\la$.  
$$\dra R^{-\frac{p+1}{(p+q)-s(q-1)}}= \la^{\frac{p+q}{2(p+q)-2s(q-1)}}.$$

Here, we get the estimate of~$\dra$ from Assumption~\ref{Cond.est} and the bound~\eqref{R.choice}. By ignoring the multiplicative constant in Assumption~\ref{Cond.est} we get the following identity from the above equation:
$$\frac{\paren{R^\dagger}^{\frac{q+p}{s(q-r)}}}{R^{\frac{r+p}{s(q-r)}}} R^{-\frac{p+1}{(p+q)-s(q-1)}}= \la^{\frac{p+q}{2(p+q)-2s(q-1)}}.$$ 

This yields
$$
R(\la) = \paren{R^{\dag}}^{\frac{(p+q)-s(q-1)}{(p+r)-s(r-1)}} \la^{\frac{s(r-q)}{2(p+r)-2s(r-1)}},\quad R>0.
$$

We get the explicit error bound from Theorem~\ref{err.upper.bound.p.1} in terms of the sample size~$m$ and~$\la$ using the above dependency~$\la \to R(\la)$.

\begin{align*}
\PP \Big\{\norm{\ip\sbrac{A(\fz)-A(\fr)}}_{\mathcal{L}^2} \leq \widetilde{C}\la^{\frac{1}{2}}&\paren{\sqrt{\frac{\mathcal{N}(\la)}{m\la}}+\paren{R^\dagger}^{\frac{p+1}{(p+r)-s(r-1)}}\la^{\frac{s(r-1)}{2(p+r)-2s(r-1)}}}\log\paren{\frac{4}{\eta}}\Big\}\geq 1-\eta,\\
\PP \Big\{\norm{\fz-\fr}_{\HH}\leq \widetilde{C}\la^{\frac{s}{2(p+1)}}&\paren{\sqrt{\frac{\mathcal{N}(\la)}{m\la}}+\paren{R^\dagger}^{\frac{p+1}{(p+r)-s(r-1)}}\la^{\frac{s(r-1)}{2(p+r)-2s(r-1)}}}^{\frac{(p+s)}{(p+1)}}\log\paren{\frac{4}{\eta}}\Big\}\geq 1-\eta.
\end{align*}

Now, we get the following error estimates using the identity~$\fz-\fp=(\fz-\fr)+(\fr-\fp)$ and the estimates of distance functions in it.

\begin{align*}
\PP \Big\{\norm{\ip\sbrac{A(\fz)-A(\fp)}}_{\mathcal{L}^2} \leq \widetilde{C}\la^{\frac{1}{2}}&\paren{\sqrt{\frac{\mathcal{N}(\la)}{m\la}}+\paren{R^\dagger}^{\frac{p+1}{(p+r)-s(r-1)}}\la^{\frac{s(r-1)}{2(p+r)-2s(r-1)}}}\log\paren{\frac{4}{\eta}}\Big\}\geq 1-\eta,\\
\PP \Big\{\norm{\fz-\fp}_{\HH}\leq \widetilde{C}\la^{\frac{s}{2(p+1)}}&\paren{\sqrt{\frac{\mathcal{N}(\la)}{m\la}}+\paren{R^\dagger}^{\frac{p+1}{(p+r)-s(r-1)}}\la^{\frac{s(r-1)}{2(p+r)-2s(r-1)}}}^{\frac{(p+s)}{(p+1)}}\log\paren{\frac{4}{\eta}}\Big\}\geq 1-\eta.
\end{align*}

By balancing the error terms, we choose the regularization parameter~$\la$ in terms of the sample size~$m$. Consequently, we get the explicit rates of convergence in terms of the sample size. 
\begin{corollary}\label{cor.err.upper.bound.gen.l}
Under the same assumptions of Theorem~\ref{err.upper.bound.p.1} and Assumption~\ref{source.cond} with~$r\leq q$,~$r+p\leq 2q$ and the a-priori choice of the regularization parameter~$\la^*=\Theta_{\mathcal{N},u}^{-1}\paren{\frac{1}{\sqrt{m}}}$ for~$\Theta_{\mathcal{N},u}(t)=\frac{t^u}{\sqrt{\mathcal{N}(t)}}$ and~$u=\frac{p+r}{2(p+r)-2s(r-1)}$, for all~$0<\eta<1$, the following error estimates holds with confidence~$1-\eta$:
\begin{align*}
\norm{\ip\sbrac{A(\fz)-A(\fp)}}_{\mathcal{L}} \leq \overline{C}\paren{\la^*}^u\log\paren{\frac{4}{\eta}}
\end{align*}
and
\begin{align*}
\norm{\fz-\fp}_{\HH}\leq &  \overline{C}\paren{\la^*}^{\frac{2u(s+p)-p}{2(p+1)}}\log\paren{\frac{4}{\eta}}
=   \overline{C}\paren{\la^*}^{\frac{sr}{2(p+r)-2s(r-1)}}\log\paren{\frac{4}{\eta}}.
\end{align*}
where~$\overline{C}$ depends on~$\ell$,~$\ell_1$,~$p$,~$q$,~$s$,~$\kappa$,~$M$,~$\Sigma$,~$\al$,~$R^\dagger$.
\end{corollary}
%We observe that the above parameter choice evidently satisfies condition~\eqref{l.la.condition}.

The effective dimension exhibits different behaviour under the different choices kernel and unknown probability measures~\cite{Lu2020}. We consider the following decay conditions on it.   

\begin{assumption}[Polynomial decay condition]\label{N(l).bound}
Assume that for some~$0<b<1$ there exists some positive constant~$C>0$ such that
\begin{equation*}
\mathcal{N}(\la):=Tr\left((\tp+\la I)^{-1}\tp\right) \leq C\la^{-b},\forall \la>0.
\end{equation*}
\end{assumption}

\begin{assumption}[Logarithmic decay condition]\label{log.decay} 
Assume that there exists some positive constant~$C>0$ such that
\begin{equation*}
\mathcal{N}(\la)\leq C\log\left(\frac{1}{\la}\right),\forall \la>0.
\end{equation*}
\end{assumption}

%In general, from Assumption~\ref{Ass:kernel}, we have
%$$\mathcal{N}(\la)\leq \norm{(\tp+\la I)^{-1}}_{\mathcal{L}(\HH)}Tr\left(\tp\right) \leq \frac{\kappa^2}{\la}.$$

%Hence under the polynomial decay condition, we get that the effective dimension of~$\HH$ behaves like power-type function. However, in general, we may not expect power-type behavior of the effective dimension. Lu et al.~\cite{Lu2020} have shown that for Gaussian kernel with the uniform sampling on~$[0,1]$, the effective dimension exhibits the log-type behavior (Assumption~\ref{log.decay}) rather than the power-type behavior (Assumption~\ref{N(l).bound}). Therefore we also consider the logarithm decay condition on the effective dimension~$\mathcal{N}(\la)$ in our analysis.

\begin{corollary}\label{err.upper.bound.p.para}
Under the same assumptions of Theorem~\ref{err.upper.bound.p.1} and Assumption~\ref{source.cond},~\ref{Cond.est},~\ref{N(l).bound} with the a-priori choice of the regularization parameter~$\la^*=m^{-\frac{1}{2u+b}}$, for all~$0<\eta<1$, the following error estimates hold with confidence~$1-\eta$:
\begin{align*}
\norm{\fz-\fp}_{\HH}\leq &\widetilde{C}\paren{\la^*}^{\frac{2u(s+p)-p}{2(p+1)}}\log\paren{\frac{4}{\eta}},\qquad u=\frac{p+q}{2(p+q)-2s(q-1)} \quad \text{for} \quad r\geq q.\\
\norm{\fz-\fp}_{\HH}\leq &\overline{C}\paren{\la^*}^{\frac{2u(s+p)-p}{2(p+1)}}\log\paren{\frac{4}{\eta}},\qquad u=\frac{p+r}{2(p+r)-2s(r-1)}\quad \text{for} \quad r\leq q,~r+p\leq 2q.
\end{align*}
\end{corollary}

\begin{corollary}\label{err.upper.bound.cor.log}
Under the same assumptions of Theorem~\ref{err.upper.bound.p.1} and Assumption~\ref{source.cond},~\ref{Cond.est},~\ref{log.decay} with the a-priori choice of the regularization parameter~$\la^*=\left(\frac{\log m}{m}\right)^{\frac{1}{2r+1}}$, for all~$0<\eta<1$, we have the following convergence rates with confidence~$1-\eta$:
\begin{align*}
\norm{\fz-\fp}_{\HH}\leq &\widetilde{C}\paren{\la^*}^{\frac{2u(s+p)-p}{2(p+1)}}\log\paren{\frac{4}{\eta}},\qquad u=\frac{p+q}{2(p+q)-2s(q-1)} \quad \text{for} \quad r \geq q.\\
\norm{\fz-\fp}_{\HH}\leq &\overline{C}\paren{\la^*}^{\frac{2u(s+p)-p}{2(p+1)}}\log\paren{\frac{4}{\eta}},\qquad u=\frac{p+r}{2(p+r)-2s(r-1)}\quad \text{for} \quad r \leq q,~r+p\leq 2q.
\end{align*}
\end{corollary}

Now, we summarize the above results with conditions. We presented the rates of convergence under the different decay conditions on the effective dimension in Corollaries~\ref{err.upper.bound.p.para},~\ref{err.upper.bound.cor.log}. In both the corollaries, first, we discuss the case when the actual smoothness is higher than the benchmark smoothness of the true solution. In this case, we get the rates of convergence corresponding to the benchmark smoothness~$q$ for~$1\leq q\leq \min(r,2+p)$,~$0<s \leq 1$. Although, the actual smoothness is higher. Second, we discuss the case when the actual smoothness is lesser than the benchmark smoothness. Here, we get the error estimates corresponding to the actual smoothness~$r$ for~$\max(1,p,r) \leq q \leq 2+p$,~$0<s\leq 1$. So, the rates are the same as what we would get by directly using the smoothness information of the true solution. At the intersection point, when~$q = r$, then both rates coincide. So, this analysis suggests that if we consider the benchmark smoothness in the appropriate range, then we would get the best rates of convergence. We emphasize that our analysis covers the oversmoothing case, i.e.,\ $r\leq 1$.  

\section*{Acknowledgements}
This research has been partially funded by Deutsche Forschungsgemeinschaft (DFG) under The Berlin Mathematics Research Center MATH+ (EXC-2046/1 - 390685689).\\
The author is grateful for fruitful discussions with Peter Math{\'e} about regularization in Hilbert Scales.

%\section{Discussion}\label{Sec:Discussion}
\appendix

\section{Probabilistic bounds}\label{Sec:prob.est}
Here, we present the standard perturbation bounds measured under the random sampling which can be obtained in~\cite{Rastogi2020}. 

\begin{proposition}\label{main.bound}
Suppose Assumption~\ref{Ass:fp}--\ref{Ass:kernel} hold true, then for~$m \in \NN$ and~$0<\eta<1$, each of the following estimates holds with the confidence~$1-\eta$,
\begin{equation*}
\Theta_{\zz}:=\norm{(\tp +\la I)^{-1/2}\sx^*\bm{\varepsilon}}_{\HH'} \leq 2\paren{\frac{\kappa M}{m\sqrt{\la}}+\sqrt{\frac{\Sigma^2\mathcal{N}(\la)}{m}}}\log\left(\frac{2}{\eta}\right),
\end{equation*}

%\begin{equation*}
%\norm{\sx^*\bm{\varepsilon}}_{\HH'} \leq 2\paren{\frac{\kappa M}{m}+\frac{\kappa\Sigma}{\sqrt{m}}}\log\left(\frac{2}{\eta}\right)
%\end{equation*}
and
\begin{equation*}
\Psi_\xx: =\norm{(\tp+\la I)^{-1/2}(\tx-\tp)}_{\mathcal{L}_2(\HH')}\leq 2\left(\frac{\kappa^2}{m\sqrt{\la}}+\sqrt{\frac{\kappa^2\mathcal{N}(\la)}{m}}\right)\log\left(\frac{2}{\eta}\right).
\end{equation*}

%\begin{equation*}
%\norm{\tx-\tp}_{\mathcal{L}_2(\HH')}\leq 2\left(\frac{\kappa^2}{m}+\frac{\kappa^2}{\sqrt{m}}\right)\log\left(\frac{2}{\eta}\right).
%\end{equation*}
\end{proposition}

Since~$\mathcal{N}(\la)$ is decreasing function of~$\la$ and~$\la\leq 1$. Therefore, from condition~\eqref{l.la.condition} we obtain,
\begin{equation*}
\mathcal{N}(1)\leq \mathcal{N}(\la) \leq  m\la
\end{equation*}
which implies that
\begin{equation*}
\frac{1}{m\sqrt{\la}}\leq \frac{1}{m\sqrt{\la}}\frac{\mathcal{N}(\la)}{\mathcal{N}(1)}=\frac{1}{\mathcal{N}(1)} \sqrt{\frac{\mathcal{N}(\la)}{m\la}}\sqrt{\frac{\mathcal{N}(\la)}{m}}\leq \frac{1}{\mathcal{N}(1)}\sqrt{\frac{\mathcal{N}(\la)}{m}}.
\end{equation*}
Now using this bound in Proposition~\ref{main.bound} we get with probability~$1-\eta$,
\begin{equation}\label{Theta.bound}
\Theta_{\zz}  \leq 2\paren{\frac{\kappa M}{\mathcal{N}(1)}+\Sigma}\sqrt{\frac{\mathcal{N}(\la)}{m}}\log\left(\frac{4}{\eta}\right)
\end{equation}
and
\begin{equation}\label{Psi.bound}
\Psi_{\xx}  \leq 2\paren{\frac{\kappa^2}{\mathcal{N}(1)}+\kappa}\sqrt{\frac{\mathcal{N}(\la)}{m}}\log\left(\frac{4}{\eta}\right).
\end{equation}

%\begin{proposition}\label{I1}
%Suppose Assumption~\ref{Ass:kernel} and the condition~\eqref{l.la.condition.k} hold true, then for~$m \in \NN$ and~$0<\eta<1$, the following estimates hold with the confidence~$1-\eta/2$,
%\begin{equation*}
%\norm{\tx-\tp}_{\mathcal{L}_2(\HH')}\leq \frac{\la}{4},
%\end{equation*}
%
%\begin{equation*}
%\norm{(\tx+\la I)^{-1}(\tp+\la I)}_{\mathcal{L}_2(\HH')}\leq 2
%\end{equation*}
%and
%\begin{equation*}
%\Gamma_{\xx}:=\norm{(\tx+\la I)^{-1/2}(\tp +\la I)^{1/2}}_{\mathcal{L}(\HH')}\leq \sqrt{2}.
%\end{equation*}
%\end{proposition}

\bibliography{library}
\bibliographystyle{plain}
\end{document}